\documentclass{article}           
\usepackage[latin1]{inputenc}     
\usepackage[T1]{fontenc}     
\usepackage{amssymb,amsmath,latexsym}
\usepackage[dvips]{graphicx}      
\usepackage[nottoc, notlof, notlot]{tocbibind}
  \usepackage[all]{xy} 

\newtheorem{theorem}{Theorem}

\newtheorem{proposition}{Proposition}

\newtheorem{corollary}{Corollary}
\newenvironment{proof}[1][Proof]{\noindent\textbf{#1.} }{\ \rule{0.5em}{0.5em}}

\title{Tensor product of $f$-rings}     
\author{Mohamed Amine BEN AMOR \\ Laboratoire LATAO \\Facult\'e des Sciences de Tunis}        

\begin{document}     
\maketitle

\begin{abstract}
In this paper we prove that $\ell$-group tensor product of archimedean $f$-rings is an $f$-ring. We will use this result to characterize multiplicative $\ell$-bimorphisms between unital $f$-rings.
\end{abstract}

\section{Introduction}
Since Martinez in \cite{MartinezTens}, constructed the $\ell$-group tensor product, several authors studied the tensor product of ordered structures, namely Fremlin in the framework of Riesz spaces (see \cite{FremlinTens}) and Banach Lattices (see \cite{FremlinBanach}).

\noindent Buskes And Van Rooij used the Fremlin tensor product to reconstruct the $\ell$-group tensor product. (see \cite{BuskesTensor}). In this end, they used the Lattice cover of $\ell$-groups (see \cite{ConradCov}).

\noindent Recently, Azouzi, Ben Amor and Jaber (see \cite{AzouziAmineJaber}) and separately Buskes and Wicksted (see \cite{Buskesfalg}) proved that the Riesz (Fremlin) tensor product of archimedean $f$-algebras is an $f$-algebra. 

\noindent In this work we will use the recent works in the framework of $f$-algebras to prove that the $\ell$-group tensor product of archimedean $f$-rings is an $f$-ring. We will use this tensor product to generalize a result of Ben Amor and Boulabiar in \cite{BoulabAmine}.

\noindent We assume that the reader is familiar with the basic concepts of the theory of lattice ordered groups ($\ell$-groups) and $f$-rings. For unexplained terminology and notations we refer to the books  \cite{AF},  \cite{BKW} and \cite{Steinberg}
\section{Tensor product of $f$-rings}

\begin{theorem}\label{cover}
If $G$ is an archimedean $f$-ring then the vector lattice cover of $G$, $R[G]$, is  an $f$-algebra.
\end{theorem}

\begin{proof}
According to \cite{ConradCov}, $R[G]$ is the $\ell$-subspace of $\overline{G^d}$ generated by $G$, where $\overline{G^d}$ is the Dedekind-MacNeille completion of the divisible hull $G^d$.

\noindent $G^d$ is obviously an archimedian $f$-ring, this with the lemma $3$ in \cite{JohnsonComp} lead to $\overline{G^d}$ is an archimedean $f$-ring.

\noindent Since $R[G]$ is the $\ell$-subspace of the $f$-ring $\overline{G^d}$ generated by the $f$-ring $G$, $R[G]$ is itself an $f$-ring and then an $f$-algebra (see Theorem $3.3$ in \cite{IsbellHenr}).
\end{proof}

\begin{corollary}\label{unitalcover}
If $G$ is a unital archimedean $f$-ring with $e_G$ as unit then the vector lattice cover of $G$, $R[G]$, is  a unital $f$-algebra with the same unit.
\end{corollary}

\begin{proof}
We proved in Theorem \ref{cover} that $R[G]$ is an $f$-ring. It remains to prove that $e_G$ is a unit in $R[G]$. The map \[ \begin{array}{cccc}
\pi_e: & R[G] &\to & R[G] \\
 & g & \mapsto & g.e_G
\end{array}
\] 
is an $\ell$-homomorphism which extends the canonical embedding of $G$ in $R[G]$. Then, according to Theorem 2 in \cite{Bleier}, $\pi_e$ is the identity and $e_G$ is the unit element of the $f$-algebra $R[G]$.
\end{proof}
\vspace{0.5cm}

\begin{theorem}\label{tensor}
Let  $G$ and $H$ be two $f$-rings, then the archimedean $\ell$-group tensor product $G \overline{\otimes}H$ is itself an $f$-ring.
\end{theorem}

\begin{proof}
In \cite{BuskesTensor}, Buskes and Van Rooij stated that the archimedean $\ell$-group tensor product $G \overline{\otimes}H$ is the $\ell$-subgroup of $R[G\otimes H]$ generated by the algebraic tensor product $G\otimes H$. But according to the same paper (Proposition $8$), $R[G\otimes H]$ is group and lattice isomorphic to $R[G]\overline{\otimes} R[H]$. So we can consider that  $G \overline{\otimes}H$ is the $\ell$-subgroup of $R[G]\overline{\otimes} R[H]$ generated by the algebraic tensor product $G\otimes H$.
Theorem \ref{cover} with \cite{AzouziAmineJaber} lead to $R[G]\overline{\otimes} R[H]$ is an $f$-algebra. Using another time Theorem $3.3$ in \cite{IsbellHenr}, we can conclude that  $G \overline{\otimes}H$ is an $f$-ring which ends the proof.
\end{proof}

\begin{corollary}\label{unitalcase}
Let  $G$ and $H$ be two unital $f$-rings  with unit element $e_G$ and $e_H$ respectively, then the archimedean $\ell$-group tensor product $G \overline{\otimes}H$ is itself a unital $f$-ring with $e_G \otimes e_H$ as unit element.  
\end{corollary}

\begin{proof}
According to corollary \ref{unitalcover}, $R[G]$ and $R[H]$ are unital $f$-algebras with $e_G$ and $e_H$ as unit element respectively. Theorem $8$ in \cite{AzouziAmineJaber} leads to $R[G]\overline{\otimes} R[H]$ is a unital $f$-algebra with $e_G \otimes e_H$ as unit element. Since $G \overline{\otimes}H$ is an $f$-subring of $R[G]\overline{\otimes} R[H]$, the result follows immediately.
\end{proof}

\section{An application}
Let $G$, $H$ and $K$ be archimedean $f$-rings. We recall that a biadditive map $\mathfrak{b}: G\times H \to K$ is said $\ell$-bimorphism if the maps: 
\[ \begin{array}{cccc}
\mathfrak{b}_1: & G &\to & K \\
 & g & \mapsto &\mathfrak{b}(g,b)
\end{array}
\]
and 
\[ \begin{array}{cccc}
\mathfrak{b}_2: & H &\to & K \\
 & h & \mapsto &\mathfrak{b}(a,h)
\end{array}
\]
are lattice homomorphisms for all $a$ in $G$ and $b$ in $H$.

\noindent An $\ell$-bimorphism $\mathfrak{b}: G\times H \to K$ is said to be multiplicative if \[\mathfrak{b}(ac,bd)=\mathfrak{b}(a,b)\mathfrak{b}(c,d)\]
for all $a$ and $c$ in $G$ and $b$ and $d$ in $H$.

\noindent Boulabiar and Toumi proved in \cite{BoulabiarToumi} that if $G$ and $H$ are archimedean $\Phi$-algebras (that is untital $f$-algebras) with $e_G$ and $e_H$ as units, and $K$ is semiprime (that is $K$ has no idempotent element) then the positive bilinear map $\mathfrak{b}$ is multiplicative if and only if $\mathfrak{b}$ is an $\ell$-bimorphism and $ \mathfrak{b}(e_G,e_H)$ is idempotent.

\noindent We will generalize this result in two directions. First we will deal with $f$-rings rather then $f$-algebras. Finally, we shall prove that the range
$f$-ring need not be reduced, which is, we believe, an important improvement.  

\noindent We pointed out that the tensor product we asked about in Theorem \ref{tensor} is the $\ell$-group tensor product that Buskes and Van Rooij studied in \cite{BuskesTensor} and earlier Martinez in \cite{MartinezTens}. The following universal propriety is still valid. Let $G$ and $H$ two archimedean $f$-rings. For any archimedean $f$-ring $K$ and every $\ell$-bimorphism $\varphi : G \times H \to K$ there exists an $\ell$-group homomorphism $\Phi : G \overline{\otimes}H \to K$ such that $\varphi (a,b) = \Phi(a \otimes b)$ for every $a$ in $G$ and $b$ in $H$.
$$
\xymatrix{
    G \times H  \ar[r]^\varphi \ar[d]_{\overline{\otimes}}  & K  \\
G \overline{\otimes}H   \ar[ru]_\Phi}
$$
The next proposition is a generalization of the Theorem $3.2$ in \cite{BoulabAmine} and it will play a key role in the generalization of Boulabiar-Toumi's theorem.
\begin{proposition}\label{positive}
Let $G$ be a unital $f$-ring with unit element $e_G$, $H$ be an archimedean $f$-ring and $T$ be a positive homomorphism between $G$ and $H$. Then $T$ is a ring homomorphism if and only if $T$ is an $\ell$-homomorphism and $T(e_G)$ is idempotent.
\end{proposition}
\begin{proof}
The "Only if" part is unchanged from Theorem $3.2$ in \cite{BoulabAmine}. Only the "if" part needs some details. Since $T$ is a ring homomorphism then $T(e_G)$ is an idempotent element and for every $a$ in $G$ we have $T(a)=T(e_G)T(a)$. This means that the range of $T$ is included in the set $T(e_G)^{\perp\perp}$ which is an $f$-ring with $T(e_G)$ as unit element (see for example Lemma $3.4$ and $3.5$ in \cite{AzouziAmine}). Now,Take $a$ and $b$ in $G$ such that $a\wedge b=0$. From \[0=T(ab)=T(a)T(b),\] and the fact that $T(e_G)^{\perp\perp}$ is reduced, we can affirm that $T(a)\wedge T(b)=0$. And we are done.
\end{proof}

We have now gathered all the ingredients we need to prove the following Theorem

\begin{theorem}
Let $G$ and $H$ be archimedean unital $f$-rings with unit element $e_G$ and $e_H$ respectively and $K$ be an archimedean $f$-ring. Let $\mathfrak{b}: G \times H \to K$ be a positive biadditive homomorphism. The following
conditions are equivalent:
\begin{itemize}
\item[i) ]$\mathfrak{b}$ is multiplicative.
\item[ii) ] $\mathfrak{b}$ is an $\ell$-bimorphism and $\mathfrak{b}(e_G,e_H)$ is idempotent.
\end{itemize}
\end{theorem}

\begin{proof}

\underline{i) $\to$ ii)}
Since $\mathfrak{b}$ is multiplicative then so are the two positive homomorphisms $\mathfrak{b}_1$ and $\mathfrak{b}_2$, where\[ \begin{array}{cccc}
\mathfrak{b}_1: & G &\to & K \\
 & g & \mapsto &\mathfrak{b}(g,e_H)
\end{array}
\]
and 
\[ \begin{array}{cccc}
\mathfrak{b}_2: & H &\to & K \\
 & h & \mapsto &\mathfrak{b}(e_G,h).
\end{array}
\]
Proposition \ref{positive} yields to $\mathfrak{b}_1$ and $\mathfrak{b}_1$ are $\ell$-homomorphisms. Which means that $\mathfrak{b}$ is an $\ell$-bimorphism. $\mathfrak{b}(e_G,e_H)$ is idempotent follows immediately.

\underline{ii) $\to$ i)}
Let $\Phi:  G \overline{\otimes}H \to K$ be the $\ell$-homomorphism such that \[\mathfrak{b}(a,b)=\Phi(a\otimes b)\] for every $a$ in $G$ and $b$ in $H$. Corollary \ref{unitalcase} yields to $G \overline{\otimes}H$ is a unital archimedean $f$-ring with $e_g\otimes e_H$ as a unit element. This, with Theorem $3.2$ in \cite{BoulabAmine} show that $\Phi$ is multiplicative. The result follows immediately since \[\mathfrak{b}(ac,bd)=\Phi(ac\otimes bd)=\Phi(a\otimes b)\Phi(c\otimes d)=\mathfrak{b}(a,b)\mathfrak{b}(c,d)\] for every $a$ and $c$ in $G$ and every $b$ and $d$ in $H$.\end{proof}

\bibliographystyle{plain}
\bibliography{mabiblio}

\begin{thebibliography}{10}

\bibitem{AF}
M.~Anderson and T.~Feil.
\newblock {\em A first course in abstract algebra}.
\newblock CRC Press, Boca Raton, FL, third edition, 2015.
\newblock Rings, groups, and fields.

\bibitem{AzouziAmineJaber}
Y.~Azouzi, M.~A. Ben~Amor, and J.~Jaber.
\newblock The tensor product of {$f$}-algebras.
\newblock {\em Quaest. Math.}, to appear.

\bibitem{AzouziAmine}
Youssef Azouzi and Mohamed~Amine Ben~Amor.
\newblock On von {N}eumann regular elements in f-rings.
\newblock {\em Algebra Universalis}, 78(1):119--124, 2017.

\bibitem{BoulabAmine}
M.~A. Ben~Amor and K.~Boulabiar.
\newblock Almost {$f$}-maps and almost {$f$}-rings.
\newblock {\em Algebra Universalis}, 69(1):93--99, 2013.

\bibitem{BKW}
A~Bigard, K.~Keimel, and S.~Wolfenstein.
\newblock {\em Groupes et anneaux r\'eticul\'es}.
\newblock Lecture Notes in Mathematics, Vol. 608. Springer-Verlag, Berlin-New
  York, 1977.

\bibitem{Bleier}
R.~D. Bleier.
\newblock Minimal vector lattice covers.
\newblock {\em Bull. Austral. Math. Soc.}, 5:331--335, 1971.

\bibitem{BoulabiarToumi}
K.~Boulabiar and M.~A. Toumi.
\newblock Lattice bimorphisms on {$f$}-algebras.
\newblock {\em Algebra Universalis}, 48(1):103--116, 2002.

\bibitem{BuskesTensor}
G.~J. H.~M. Buskes and A.~C.~M. van Rooij.
\newblock The {A}rchimedean {$l$}-group tensor product.
\newblock {\em Order}, 10(1):93--102, 1993.

\bibitem{Buskesfalg}
G.~J. H.~M. Buskes and A.~W. Wickstead.
\newblock Tensor products of {$f$}-algebras.
\newblock {\em Mediterr. J. Math.}, 14(2):Art. 63, 10, 2017.

\bibitem{ConradCov}
P.~F. Conrad.
\newblock Minimal vector lattice covers.
\newblock {\em Bull. Austral. Math. Soc.}, 4:35--39, 1971.

\bibitem{FremlinTens}
D.~H. Fremlin.
\newblock Tensor products of {A}rchimedean vector lattices.
\newblock {\em Amer. J. Math.}, 94:777--798, 1972.

\bibitem{FremlinBanach}
D.~H. Fremlin.
\newblock Tensor products of {B}anach lattices.
\newblock {\em Math. Ann.}, 211:87--106, 1974.

\bibitem{IsbellHenr}
M.~Henriksen and J.~R. Isbell.
\newblock Lattice-ordered rings and function rings.
\newblock {\em Pacific J. Math.}, 12:533--565, 1962.

\bibitem{JohnsonComp}
D.~G. Johnson.
\newblock The completion of an archimedean {$f$}-ring.
\newblock {\em J. London Math. Soc.}, 40:493--496, 1965.

\bibitem{MartinezTens}
J.~Martinez.
\newblock Tensor products of partially ordered groups.
\newblock {\em Pacific J. Math.}, 41:771--789, 1972.

\bibitem{Steinberg}
S.~A. Steinberg.
\newblock {\em Lattice-ordered rings and modules}.
\newblock Springer, New York, 2010.

\end{thebibliography}

\end{document}